\documentclass[a4paper,10pt,leqno]{amsart}
\usepackage{amsmath,amsfonts,amsthm,amssymb,dsfont}
\usepackage[alphabetic]{amsrefs}
\usepackage[OT4]{fontenc}
\usepackage{enumerate}
\usepackage{mathrsfs}
\usepackage{enumerate, xspace}
\usepackage[pdftex]{graphicx}
\usepackage{color}

\long\def\symbolfootnote[#1]#2{\begingroup
\def\thefootnote{\fnsymbol{footnote}}\footnote[#1]{#2}\endgroup}

\newtheorem{theorem}{Theorem}[section]
\newtheorem{lemma}[theorem]{Lemma}
\newtheorem{thm}[theorem]{Theorem}
\newtheorem{sublemma}[theorem]{Sublemma}
\newtheorem{prop}[theorem]{Proposition}
\newtheorem{cor}[theorem]{Corollary}
\newtheorem{question}{Question}

\theoremstyle{definition}
\newtheorem{rem}[theorem]{Remark}
\newtheorem{defin}[theorem]{Definition}

\newcommand{\id}{\mathrm{id}}
\newcommand{\N}{\mathbf{N}}
\newcommand{\R}{\mathbf{R}}
\newcommand{\pla}{\k{a}}

\begin{document}

\title{2-dimensional Coxeter groups are biautomatic}

\author[Z.~Munro]{Zachary Munro}
	
\address{
Department of Mathematics and Statistics,
McGill University,
Burnside Hall,
805 Sherbrooke Street West,
Montreal, QC,
H3A 0B9, Canada}

\email{zachary.munro@mcgill.ca}

\author[D.~Osajda]{Damian Osajda$^{\dag}$}
\address{Instytut Matematyczny,
	Uniwersytet Wroc\l awski\\
	pl.\ Grun\-wal\-dzki 2/4,
	50--384 Wroc\-{\l}aw, Poland}
\address{Institute of Mathematics, Polish Academy of Sciences\\
	\'Sniadeckich 8, 00-656 War\-sza\-wa, Poland}
\email{dosaj@math.uni.wroc.pl}
\thanks{$\dag \ddag$ Partially supported by (Polish) Narodowe Centrum Nauki, UMO-2018/30/M/ST1/00668.}

\author[P.~Przytycki]{Piotr Przytycki$^{\ddag}$}
	
\address{
Department of Mathematics and Statistics,
McGill University,
Burnside Hall,
805 Sherbrooke Street West,
Montreal, QC,
H3A 0B9, Canada}

\email{piotr.przytycki@mcgill.ca}

\thanks{$\ddag$ Partially supported by NSERC and AMS}

\maketitle

\begin{abstract}
\noindent
Let $W$ be a $2$-dimensional Coxeter group, that is, one with $\frac{1}{m_{st}}+\frac{1}{m_{sr}}+\frac{1}{m_{tr}}\leq 1$ for all triples of distinct $s,t,r\in S$. We prove that $W$ is biautomatic. We do it by showing that a natural geodesic language is regular (for arbitrary~$W$), and satisfies the fellow traveller property. As a consequence, by the work of Jacek \'{S}wi\pla tkowski, groups acting properly and cocompactly on buildings of type~$W$ are also biautomatic. We also show that the fellow traveller property for the natural language fails for $W=\widetilde{A}_3$.
\end{abstract}

\section{Introduction}
\label{sec:intro}

A \emph{Coxeter group} $W$ is a group generated by a finite set
$S$ subject only to relations $s^2=1$ for $s\in S$ and
$(st)^{m_{st}}=1$ for $s\neq t\in S$, where $m_{st}=m_{ts}\in \{2,3,\ldots,\infty\}$.
Here the convention is that $m_{st}=\infty$ means that we do
not impose a relation between $s$ and $t$. We say that $W$ is
\emph{$2$-dimensional} if for any triple of distinct elements $s,t,r\in S$,
the group $\langle s,t,r\rangle$ is infinite. In other words,
$\frac{1}{m_{st}}+\frac{1}{m_{sr}}+\frac{1}{m_{tr}}\leq 1$.

Consider an arbitrary group $G$ with a finite symmetric generating set
$S$. For $g\in G$, let $\ell(g)$ denote
the \emph{word length} of $g$, that is, the minimal number $n$
such that $g=s_1\cdots s_n$ with $s_i\in S$ for $i=1,\ldots,
n$. Let $S^*$ denote the set of all words over $S$.
If $v\in S^*$ is a word of length $n$, then by $v(i)$ we
denote the prefix of $v$ of length $i$ for $i=1,\ldots, n-1$,
and the word $v$ itself for $i\geq n$. For $1\leq i\leq j\leq
n$ by $v(i,j)$ we denote the subword of $v(j)$ obtained by
removing $v(i-1)$. For a word $v\in S^*$, by $\ell(v)$ we denote
the word length of the group element that $v$ represents.

We say that $G$ is
\emph{biautomatic} if there exists a regular language $\mathcal
L\subset S^*$ (see Section~\ref{sec:regularity} for the definition of regularity) and a constant $C>0$ satisfying the following conditions (see \cite[Lem~2.5.5]{E}).
\begin{enumerate}[(i)]
\item For each $g\in G$, there is a word in $\mathcal L$
    representing $g$.
\item For each $s\in S$ and $g,g'\in G$ with $g'=gs,$ and each $v,v'\in \mathcal L$
    representing $g,g'$, for
    all $i\geq 1$ we have $\ell(v(i)^{-1}v'(i))\leq C$.
\item For each $s\in S$ and $g,g'\in G$ with $g'=sg,$ and each $v,v'\in \mathcal L$
    representing $g,g'$, for
    all $i\geq 1$ we have $\ell(v(i)^{-1}s^{-1}v'(i))\leq C$.
\end{enumerate}

Our paper concerns the two following well-known open questions (see e.g.\ \cite[\S6.6]{FHT}).

\begin{question}
	\label{q:1}
	Are Coxeter groups biautomatic?
\end{question}

\begin{question}
	\label{q:2}
	Are groups acting properly and cocompactly on $2$-dimensional $\mathrm{CAT}(0)$ spaces biautomatic?
\end{question}

All Coxeter groups are known to be automatic (i.e.\ having a regular language satisfying (i) and (ii)) by \cite{BH}.
Biautomaticity has been established only in special cases: \cite{E} (Euclidean and hyperbolic), \cite{NiRe2003} (right-angled),
\cite{Bahls2006} and \cite{CaMu2005} (no Euclidean reflection triangles), \cite{Cap2009} (relatively hyperbolic).
\medskip

Question~\ref{q:2} is widely open. The assumption of $2$-dimensionality is essential, since recently Leary--Minasyan \cite{LeMi} constructed a group acting properly and cocompactly on a $3$-dimensional $\mathrm{CAT}(0)$ space that is not biautomatic. Even in the case of $2$-dimensional buildings, except right-angled and hyperbolic cases, the answer was known
only in particular instances, e.g.\  for many (but not all) proper cocompact actions on Euclidean buildings by \cite{GeSh1}, \cite{GeSh2},
\cite{CaSh1995}, \cite{Nos2000}, and \cite{Sw}.

\smallskip

To define a convenient language, we need the following.
Let $W$ be an arbitrary Coxeter group. For $g\in W$, we denote
by $T(g)\subseteq S$ the set of $s\in S$ satisfying $\ell(gs)<\ell(g)$.
By \cite[Thm~2.16]{R}, the group $\langle T(g)\rangle$ is finite. By $w(g)$ we denote the
longest element in $\langle T(g) \rangle$ (which is unique by \cite[Thm~2.15(iii)]{R}, and consequently it is an involution). Let $\Pi(g)=gw(g)$. By
\cite[Thm~2.16]{R}, we have $\ell(\Pi(g))+\ell(w(g))=\ell(g)$.

We define the \emph{standard language} $\mathcal L\subset S^*$ for $W$ inductively
in the following way. Let $v\in S^*$ be a word of length $n$.
If $v$ represents the identity element of $W$, then $v\in \mathcal L$ if and only if $v$ is the empty word. Otherwise,
let $g\in W$ be the group element represented by $v$ and let $k=\ell(w(g))$. We declare $v\in \mathcal L$ if and
only if $v(n-k)\in \mathcal L$ and $v(n-k+1,n)$ represents $w(g)$. In particular, $v(n-k)$ represents $\Pi(g)$. It follows inductively that $n=\ell(g)$. Such a language is called \emph{geodesic}. Note that the standard language satisfies part~(i) of the definition of biautomaticity.

The paths in $W$ formed by the words in the standard language generalise the normal cube paths for $\mathrm{CAT}(0)$ cube complexes \cite[\S3]{NR} used to prove biautomaticity for right-angled (or, more generally, cocompactly cubulated) Coxeter groups \cite{NiRe2003}. Our main result is the following.

\begin{thm}
\label{thm:main}
If $W$ is a $2$-dimensional Coxeter group, then
it is biautomatic with $\mathcal L$ the standard language.
\end{thm}

Since the standard language is geodesic and preserved by the automorphisms of~$W$ stabilising~$S$, by \cite[Thm~6.7]{Sw} we have the following immediate consequence.

\begin{cor}
\label{cor:buildings}
Let $G$ be a group acting properly and cocompactly on a building of type $W,$ where $W$ is a $2$-dimensional Coxeter group. Then $G$ is biautomatic.
\end{cor}

One element of our proof of Theorem~\ref{thm:main} is:

\begin{thm}
\label{lem:regular}
Let $W$ be a Coxeter group. Then its standard language is
regular.
\end{thm}

In other words, the regularity and part (i) of the definition of biautomaticity are
satisfied for any Coxeter group $W$. However, it is not so
with part (ii). The \emph{$\widetilde A_3$ Euclidean group} is the Coxeter
group with
$S=\{p,r,s,t\}, m_{pr}=m_{rs}=m_{st}=m_{tp}=3,
m_{ps}=m_{rt}=2$.

\begin{thm}
\label{thm:second}
If $W$ is the $\widetilde A_3$ Euclidean group, then its
standard language does not satisfy part (ii) in the definition
of biautomaticity.
\end{thm}

Note, however, that by \cite[Cor~4.2.4]{E}, all Euclidean groups, in particular~$\widetilde A_3$, are biautomatic (with a different language).

\medskip

\noindent \textbf{Organisation.} In Section~\ref{sec:prel} we review the basic properties of Coxeter groups. In Section~\ref{sec:regularity} we prove Theorem~\ref{lem:regular}. For $2$-dimensional $W$, we verify parts~(iii) and~(ii) of the definition of biautomaticity in Sections~\ref{sec:partiii} and~\ref{sec:main}. This completes the proof of Theorem~\ref{thm:main}. We finish with the proof of Theorem~\ref{thm:second} in Section~\ref{sec:dim3}.

\medskip

\noindent \textbf{Acknowledgement.} We thank the referee for many helpful suggestions.

\section{Preliminaries}
\label{sec:prel}

By $X^1$ we denote the \emph{Cayley graph} of $W$, that is, the graph with vertex set $X^0=W$ and with edges joining
each $g\in W$ with $gs$, for $s\in S$. We call such an edge an
\emph{$s$-edge}. We call $gs$ the
\emph{$s$-neighbour} of $g$.

For $r\in W$ a conjugate of an element
of $S$, the \emph{wall} $\mathcal W_r$ of $r$ is the fixed point set of $r$ in $X^1$. We call $r$ the \emph{reflection} in $\mathcal W_r$ (for fixed $\mathcal W_r$ such $r$ is unique). If a midpoint of an edge $e$
belongs to a wall $\mathcal W$, then we say that $\mathcal W$ is \emph{dual} to~$e$ (for fixed~$e$ such a wall is unique).
We say that $g\in W$ is \emph{adjacent} to a wall~$\mathcal W$, if $\mathcal W$ is dual to an edge incident to $g$.
Each wall $\mathcal W$ separates $X^1$ into two components, and
a geodesic edge-path in $X^1$ intersects $\mathcal W$ at most once \cite[Lem~2.5]{R}.

For $T\subseteq S$, each coset $g\langle T\rangle\subseteq X^0$ for $g\in W$
is a \emph{$T$-residue}. A geodesic edge-path in $X^1$ with endpoints in a residue $R$ has all its vertices in $R$ \cite[Lem~2.10]{R}. We say that a wall $\mathcal W$ \emph{intersects}
a residue $R$ if $\mathcal W$ separates some elements of $R$. Equivalently, $\mathcal W$ is dual to an edge
with both endpoints in $R$.

\begin{thm}[{\cite[Thm~2.9]{R}}]
\label{thm:proj}
Let $W$ be a Coxeter group. Any residue $R$ of~$X^0$ contains a unique element $h$ with minimal $\ell(h)$. Moreover, for any $g\in R$ we have $\ell(h)+\ell(h^{-1}g)=\ell(g)$.
\end{thm}

As introduced in Section~\ref{sec:intro}, for $g\in W$ we denote
by $T(g)\subseteq S$ the set of $s\in S$ satisfying $\ell(gs)<\ell(g)$.
Let $R$ be the $T(g)$-residue containing $g$.
By \cite[Thm~2.16]{R}, the group $\langle T(g)\rangle$ is finite and, for $w(g)$ the
longest element in $\langle T(g) \rangle$, the unique element $h\in R$ from Theorem~\ref{thm:proj} is $\Pi(g)=gw(g)$. In particular,
we have $\ell(\Pi(g))+\ell(w(g))=\ell(g)$. Note that if $W$ is $2$-dimensional, then for each $g\in W$ we have $|T(g)|=1$ or $2$. See Figure~\ref{f:f0} for an example where $S=\{s,t,r\}$ with $m_{st}=2, m_{sr}=m_{tr}=4$, and $g=strst$.

\begin{figure}[h!]
\begin{center}
\includegraphics[scale=0.6]{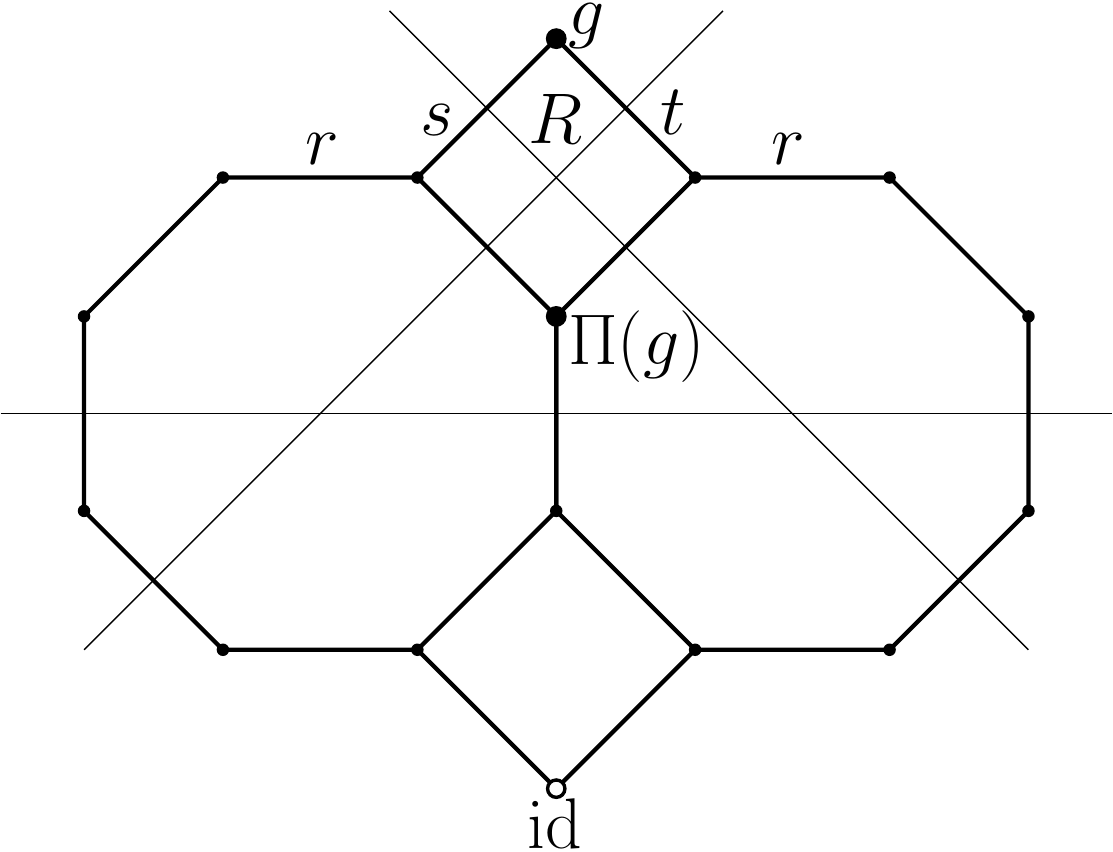}
\end{center}
\caption{$T(g)=\{ s,t\}$ and $w(g)=st=ts$.}
\label{f:f0}
\end{figure}

For $g\in W$, let $\mathcal W(g)$ be the set of walls $\mathcal W$ in $X^1$ that separate $g$ from the identity element $\id\in W$ and such that there is no wall $\mathcal W'$ separating $g$ from $\mathcal W$.

\begin{rem}
\label{rem:WandR}
Let $g\in W$ and let $R$ be the $T(g)$-residue containing $g$. Since $R$ is finite, all the walls intersecting $R$ belong to $\mathcal W(g)$. However, there might be walls in $\mathcal W(g)$ that do not intersect $R$. See Figure~\ref{f:f0} for an example, where we indicated all three walls of $\mathcal W(g)$ for $g=strst$.
\end{rem}

By the following Parallel Wall Theorem, there exists a bound on the distance in~$X^1$ between $g$ and each of the walls of $\mathcal W(g)$.

\begin{thm}[{\cite[Thm~2.8]{BH}}]
\label{thm:parallel}
Let $W$ be a Coxeter group. There is a constant $Q=Q(W)$ such that for any $g\in W$ and a wall $\mathcal W$ at distance $> Q$ from $g$ in $X^1$, there is a wall $\mathcal W'$ separating $g$ from $\mathcal W$.
\end{thm}

By $X$ we denote the \emph{Cayley complex} of $W$. It is the
piecewise Euclidean $2$-complex with $1$-skeleton $X^1$, all
edges of length $1$, and a regular $2m_{st}$-gon spanned on
each $\{s,t\}$-residue with $m_{st}<\infty$. If $W$ is
$2$-dimensional, then $X$ is $\mathrm{CAT}(0)$, see \cite[\S II.5.4]{BHa} and the link condition in \cite[\S II.5.6]{BHa}. ($X$ coincides then with the `Davis complex' of $W$.)
Walls in $X^1$ extend to (convex) walls in $X,$ which still separate $X$.

We will consider the action of $G$ on $X^0=G$ by left multiplication. This induces obvious actions of $G$ on $X^1, X$ and the set of walls.
\section{Regularity}
\label{sec:regularity}

A \emph{finite state automaton over $S$} (FSA) is a finite directed graph $\Gamma$ with vertex set $V$, edge set $E\subseteq V\times V$, an edge labeling $\phi\colon E\to \mathcal P(S^*)$ (the power set of $S^*$), a distinguished set of \emph{start states} $S_0\subseteq V$, and a distinguished set of \emph{accept states} $F\subseteq V$. A word $v\in S^*$ is \emph{accepted by $\Gamma$} if there exists a decomposition $v=v_0\cdots v_m$ of $v$ into subwords and an edge-path $e_0\cdots e_m$ in $\Gamma$ such that $e_0$ has initial vertex in $S_0$, $e_m$ has terminal vertex in~$F$, and $v_i\in \phi(e_i)$ for each $i=0,\ldots,m$. A subset of $S^*$ is a \emph{regular language} if it is the set of accepted words for some FSA over $S$.

The proof of the regularity of the standard language relies on Theorem~\ref{thm:parallel} and the following lemma.

\begin{lemma}
\label{lem:append}
Let $W$ be a Coxeter group. Let $g\in W$, let $T\subseteq S$ be such that
$\langle T \rangle$ is finite, and let $w$ be the longest
element in $\langle T \rangle$. Then $T(gw)=T$ if and only if
\begin{enumerate}[(i)]
\item
$T$ is disjoint from $T(g)$, and
\item
for each
$t\in S\setminus T$,
the wall dual to $(gw,gwt)$ does not lie in $\mathcal W(g)$.
\end{enumerate}
\end{lemma}

Note that for $g\in W$ and $s\in S$, the wall dual to $(g,gs)$ lies in $\mathcal W(g)$ if and only if it separates $g$ from $\id$. Consequently, condition~(i) could be written equivalently as: for each
$t\in T$, the wall dual to $(g,gt)$ does not lie in $\mathcal W(g)$.

\begin{proof}[Proof of Lemma~\ref{lem:append}]
Suppose first $T(gw)=T$. Then, for $R$ the $T$-residue containing $gw$, by the discussion after Theorem~\ref{thm:proj}, the unique element $h\in R$ with minimal $\ell(h)$ is $g$. Thus for each $t\in T$ we
have $\ell(gt)> \ell(g)$ and so condition~(i) holds.
Furthermore, for $t\in S\setminus T$, the wall~$\mathcal W$
dual to the edge $(gw,gwt)$ does not separate $gw$
from $\id$. Additionally, the wall $\mathcal W$ cannot separate $gw$ from $g$: If it did, then after
conjugating by $(gw)^{-1}$, the reflection in $\mathcal W$ could become simultaneously the generator~$t$ and a word in the elements of $T$, contradicting $t\in S\setminus T$ by \cite[Lem~2.1(ii)]{R}. Thus $\mathcal W$ does not separate $g$ from $\id$, and so condition~(ii) holds.

Conversely, suppose that we have $T\subseteq S$ 
satisfying conditions~(i) and~(ii).
Then, by condition~(i), for $R$ the $T$-residue containing $g$, we have that the minimal word length element $h\in R$ from Theorem~\ref{thm:proj} coincides with $g$, and so the element of $R$ of maximal word length is $gw$. Consequently, we have $T(gw)\supseteq T$.
Suppose, for a contradiction, that there is
$t\in T(gw)\setminus T$. Then the wall $\mathcal W$
dual to the edge $(gw,gwt)$ separates $gw$ from $\id$. The same argument as in the previous paragraph implies that $\mathcal W$ does not separate $gw$ from~$g$, so it separates $g$ from $\id$. Furthermore, if a wall $\mathcal W'$ separated $\mathcal W$ from $g$, then $\mathcal W'$ would also have to separate $gw$ from $g$, contradicting $\ell(g)+\ell(w)=\ell(gw)$. Consequently, $\mathcal W\in \mathcal W(g)$, which contradicts condition~(ii).
\end{proof}

We now define an FSA $\Gamma$ over $S$ that will accept exactly the standard language.

\begin{defin}
\label{def:reg}
Let $Q$ be the constant from Theorem~\ref{thm:parallel}. For $g\in W$, let $\mathcal U_Q(g)$ be the set of walls in $X^1$ intersecting the closed ball in $X^1$ of radius $Q$ centred at~$g$. By Theorem~\ref{thm:parallel}, we have $\mathcal W(g)\subseteq \mathcal U_Q(g)$.

Consider the set $\hat V$ of pairs of the form $(g, \mathcal U)$, where $g\in W$, and $\mathcal U$ is a subset of $\mathcal U_Q(g)$.
We define an equivalence relation $\sim$ on $\hat V$ by $(g,\mathcal U)\sim (h, \mathcal U')$ if $\mathcal U'=hg^{-1}\mathcal U$. We take the vertices of our FSA $\Gamma$ to be $V=\hat V/\sim$. To lighten the notation, we denote the equivalence class of $(g,\mathcal U)$ by $[g,\mathcal U]$.

In any equivalence class of $\sim$, there is exactly one representative of the form $(\id, \mathcal U)$. Suppose that we have
$T\subseteq S$ such that $\langle T\rangle$ is finite. Let $w$ be the longest element of $\langle T\rangle$. If
\begin{enumerate}[(i)]
\item
for each $t\in T,$ the wall dual to $(\id,t)$ lies outside $\mathcal U$, and
\item
for each $t\in S\setminus T,$ the wall dual to $(w,wt)$ lies outside $\mathcal U$,
\end{enumerate}
then we put an edge $e$ in $\Gamma$ from $[\id, \mathcal U]$ to $[w, \mathcal U']$, where
$\mathcal U'$ is defined as the set of walls in $\mathcal U_Q(w)$ that
\begin{enumerate}[(a)]
\item
lie in $\mathcal U$ or intersect the residue $\langle T\rangle$, and
\item
are not separated from $w$ by a wall satisfying~(a).
\end{enumerate}

We let the label~$\phi(e)$ to be the set of all minimal length words representing~$w$.
We let all states be accept states of $\Gamma$ and let the set of start states $S_0$ contain only~$[\id, \emptyset]$.
\end{defin}

\begin{proof}[Proof of Theorem~\ref{lem:regular}]
Let $\Gamma$ be the FSA from Definition~\ref{def:reg}, and let $\mathcal L$ be the standard language.
We argue inductively on $j\geq 0$ that, among the words $v\in S^*$ of length $\leq j$,
\begin{itemize}
\item
$\Gamma$ accepts exactly the words in $\mathcal L$, and
\item
the accept state of each such word $v$ is $[g,\mathcal W(g)],$ where $v$ represents $g$.
\end{itemize}
This is true for $j=0$ by our choice of $S_0$. Now let $n>0$ and suppose that we have verified the inductive hypothesis for all $j<n$. Let $v$ be a word in $S^*$ of length $n$.

Suppose first that $v$ is a word in $\mathcal L$ representing $g\in W$. By the definition of~$\mathcal L$, for $k=\ell(w(g))$, we have $v(n-k)\in \mathcal L$. Moreover, $v(n-k+1,n)$ represents~$w(g)$. By the inductive hypothesis, $\Gamma$ accepts $v(n-k)$. Furthermore, $v(n-k)$ labels some edge-path in $\Gamma$ from $S_0$ to $[\Pi(g),\mathcal W(\Pi(g))]$.
Let $T=T(g)$ and $w=w(g)$. By Lemma~\ref{lem:append}, applied replacing $g$ with~$\Pi(g)$, we have that
\begin{enumerate}[(i)]
\item for each $t\in T$, the wall dual to $(\Pi(g),\Pi(g)t)$ does not lie in $\mathcal W(\Pi(g))$, and
\item for each $t\in S\setminus T,$ the wall dual to $(g,gt)$ does not lie in $\mathcal W(\Pi(g))$.
\end{enumerate}
Thus, translating by $\Pi(g)^{-1}$, we see that $\Gamma$ has an edge from $[\id,\Pi(g)^{-1}\mathcal W(\Pi(g))]=[\Pi(g),\mathcal W(\Pi(g))]$ to $[w,\mathcal U']=[g,\Pi(g)\mathcal U']$, labelled by $v(n-k+1,n)$, and so $\Gamma$ accepts~$v$. Furthermore,
by conditions~(a) and~(b) in Definition~\ref{def:reg}, we have that $\Pi(g)\mathcal U'$ consists of walls of $\mathcal U_Q(g)$ that lie in $\mathcal W(\Pi(g))$ or intersect the residue $g\langle T\rangle$ and are not separated from $g$ by any other such wall. Since $\mathcal W(g)\subseteq \mathcal U_Q(g)$, this implies $\Pi(g)\mathcal U'=\mathcal W(g)$.

Conversely, let $v$ be accepted by $\Gamma$ and suppose that $v=v_0\cdots v_m$ as in the definition of an accepted word.
By the inductive hypothesis, the word $v_0\cdots v_{m-1}$ belongs to $\mathcal L$ and represents $g\in W$ such that $e_m$ starts at
$[g,\mathcal W(g)]=[\id, g^{-1}\mathcal W(g)]$. By the definition of the edges, $v_{m}$ represents the longest element~$w$ in some finite~$\langle T \rangle$, and $\mathcal U=g^{-1}\mathcal W(g)$ satisfies conditions~(i) and~(ii) in Definition~\ref{def:reg}.
Translating by $g$, we obtain that $g$ and $T$ satisfy conditions~(i) and~(ii) of Lemma~\ref{lem:append}. Consequently, we have $T=T(gw)$, and so $v$ belongs to~$\mathcal L$.
\end{proof}

\section{$g$ and $sg$}
\label{sec:partiii}
\begin{lemma}
\label{lem:part3}
Let $W$ be a $2$-dimensional Coxeter group. Then its standard language satisfies part (iii) of the definition of biautomaticity.
\end{lemma}

We will need the following.

\begin{sublemma}
\label{sub}
Let $W$ be a $2$-dimensional Coxeter group.
There is a constant $D=D(W)$ such that for any wall $\mathcal W$ adjacent to $\id$, any $f\in W$ adjacent to $\mathcal W$, and any vertices $h,h'\in W$ on geodesic edge-paths from $\id$ to $f$ satisfying $\ell(h)=\ell(h')$, we have $\ell(h^{-1}h')< D$.
\end{sublemma}

\begin{proof}
Let $Q=Q(W)$ be the constant from Theorem~\ref{thm:parallel}. Suppose
that $h,h'\in W$ lie on geodesic edge-paths~$\gamma,\gamma'$ from $\id$ to $f$ and satisfy $\ell(h)=\ell(h')$.
Note that each vertex $g\in W$ of $\gamma$ lies at distance $\leq Q$ from $\mathcal W$ in $X^1$,
since otherwise there would be a wall $\mathcal W'$ separating $g$
from~$\mathcal W$, and so $\mathcal W'$ would intersect $\gamma$
at least twice.

Since $W$ is $2$-dimensional, we have that $X$ is a $\mathrm{CAT}(0)$ space, with path-metric that we denote $|\cdot, \cdot|$,
and the extension of~$\mathcal W$ to $X$ (for which we keep the same notation) is a convex tree. Let $x,y\in \mathcal W$ be the midpoints of the edges dual to $\mathcal W$ incident to $\id,f$, respectively. Let $N(\mathcal W)$ be the closed $Q$-neighbourhood of $\mathcal W$ in $X$, w.r.t.\ the $\mathrm{CAT}(0)$ metric. Note that $N(\mathcal W)$ is quasi-isometric to $\mathcal W$, so in particular $N(\mathcal W)$ is Gromov-hyperbolic (for definition, see e.g.\ \cite[III.H.1.1]{BHa}). Moreover, since $X$ and $X^1$ are quasi-isometric, we have that $\gamma\subset N(\mathcal W)$ is a $(\lambda,\epsilon)$-quasigeodesic (for definition, see \cite[I.8.22]{BHa}), where the constants $\lambda,\epsilon$ depend only on $W$. Consequently, by the stability of quasi-geodesics \cite[III.H.1.7]{BHa}, for a constant $C=C(W)$, there is a point $z$ on the geodesic from $x$ to $y$ with $|h,z|\leq C$. Analogously, there is a vertex $h''$ on $\gamma'$ with $|z,h''|\leq C$, and so $|h,h''|\leq 2C$.

Thus, since $X$ and $X^1$ are quasi-isometric, there is a constant $D=D(W)$ with $\ell(h^{-1}h'')< \frac{D}{2}$. By the triangle inequality in $X^1$, we have $|\ell(h)-\ell(h'')|< \frac{D}{2}$. Thus, by $\ell(h)=\ell(h')$, the distances on $\gamma'$ from $h''$ and $h'$ to $\id$ differ by less than $\frac{D}{2}$. Consequently,
we have $\ell(h''^{-1}h')< \frac{D}{2}$, and so $\ell(h^{-1}h')< D$, as desired.
\end{proof}

\begin{proof}[Proof of Lemma~\ref{lem:part3}] Let $\mathcal L$ be the standard language. Let $D$ be the constant from Sublemma~\ref{sub}. Let $K$ be the maximal word length of the longest element of a finite
$\langle T\rangle$ over all $T\subseteq S$, and let $C=\max\{K,D\}$.

We prove part~(iii) of the definition of biautomaticity
inductively on $\ell(g)$, where we assume without loss of generality
$\ell(sg)>\ell(g)$. If $g=\id$, then there is nothing to prove.
Suppose now $g\neq \id$, and let $\mathcal W$ be the wall in $X^1$ dual to the $s$-edge incident to $\id$. Let $v,v'\in \mathcal L$ represent $g,sg,$ respectively.

Assume first that $g$ is not adjacent to $\mathcal W$. Let $\mathcal W'$ be a wall adjacent to $g$ separating $g$ from $\id$. Then $\mathcal W'$
also separates $g$ from~$s$. Consequently, $s\mathcal W'$
separates $sg$ from $\id$. Conversely, if a wall $\mathcal W'$ is adjacent to $sg$ and separates $sg$ from $\id$, then it also
separates $sg$ from $s$, and so $s\mathcal W'$ separates $g$ from $\id$. Consequently,
$T(sg)=T(g)$ and so $w(g)=w(sg)$, hence $\Pi(sg)=s\Pi(g)$. In
other words, for $k=\ell(w(g))$, the
words $v'(\ell(sg)-k)$ and $sv(\ell(g)-k)$ represent the same element $s\Pi(g)$ of $W$. Then
part (iii) of the definition of biautomaticity for $g$ follows inductively
from part~(iii) for $\Pi(g)$, for $i< \ell(sg)-k$, or from the definition of $K$, for $i\geq \ell(sg)-k$.

Secondly, assume that $g$ is adjacent to $\mathcal W$. Then $(g,sg)$ is an edge of
$X^1$. Let $f=sg$ and for $0\leq i\leq \ell(g)$ let $h,h'$ be the elements of $W$
represented by $sv(i)$ and $v'(i+1)$. Then, by the definition of
$D$, we have $\ell(v(i)^{-1}sv'(i+1))<D$, as desired.
\end{proof}

\section{$g$ and $gs$}
\label{sec:main}

For $g\in W$ and $k\geq 0$, we set $\Pi^k(g)=\overbrace{\Pi\circ
\cdots \circ\Pi}^{k}(g)$. The main result of this section is the following.

\begin{prop}
\label{prop:main}
Let $W$ be a $2$-dimensional Coxeter group.
Let $g,g'\in W$ be such that $g'\in g \langle s,t \rangle$ for
some $s,t\in S$ with $m_{st}<\infty$ (possibly $s=t$). Then there are $0\leq k,k'\leq 3$ with
$k+k'>0$, such that $\Pi^{k'}(g')\in \Pi^k(g)\langle p,r\rangle$
for some $p,r\in S$ with $m_{pr}<\infty$ (possibly $p=r$).
\end{prop}

We obtain the following consequence, which together with Theorem~\ref{lem:regular} and Lemma~\ref{lem:part3} completes the proof of Theorem~\ref{thm:main}.

\begin{cor}
\label{cor:fellow_travel}
Let $W$ be a $2$-dimensional Coxeter group. Then its standard language satisfies part (ii) of the definition of
biautomaticity.
\end{cor}
\begin{proof}
As before, let $K$ be the maximal word length of the longest element of a finite
$\langle T\rangle$ over all $T\subseteq S$. Assume without loss of generality
$\ell(gs)>\ell(g)$.

Let $0\leq i\leq \ell(g)$. By Proposition~\ref{prop:main}, there is $0\leq j\leq \ell(g)$ with $|j-i|\leq \frac{3K}{2}$ and $0\leq i'\leq \ell(g)+1$ such that $v(j)$ and $v'(i')$ represent elements of $W$ in a common finite residue. Consequently, we have $\ell\big(v(j)^{-1}v'(i')\big)\leq K,$ and so in particular $|j-i'|\leq K$. Therefore $\ell\big(v(i)^{-1}v'(i)\big)\leq |i-j|+\ell\big(v(j)^{-1}v'(i')\big)+|i'-i|\leq 5K$.
\end{proof}

In the proof of Proposition~\ref{prop:main} we will use the following \emph{truncated piecewise Euclidean structure}
on the barycentric subdivision $X'$ of the Cayley complex $X$ of~ $W$. Consider the function $q\colon\N_{\geq 2}\to \N_{\geq 2}$, defined as

\[q(m)=\begin{cases}
m, & \text{for } m=2,3, \\
4, & \text{for } m=4,5, \\
    6, & \text{for } m\geq 6.
\end{cases}
\]

Note that each triangle $\sigma$ of $X'$ is a triangle in the barycentric subdivision of a
regular $2m$-gon of $X$ spanned on an $\{s,t\}$-residue with $m_{st}=m<\infty$. Consequently, in the usual piecewise Euclidean structure, $\sigma$ has angles $\frac{\pi}{2m},\frac{\pi}{2},\big(1-\frac{1}{m}\big)\frac{\pi}{2}$. Moreover, the edge opposite to $\frac{\pi}{2m}$ is half of the edge of $X^1$, so it has length $\frac{1}{2}$. In the truncated piecewise Euclidean structure, we choose a different metric on $\sigma$, namely that of a triangle in the barycentric subdivision of
a regular $2q(m)$-gon. More precisely, the angles of $\sigma$ are $\frac{\pi}{2q(m)},\frac{\pi}{2},\big(1-\frac{1}{q(m)}\big)\frac{\pi}{2}$, while the length of the edge opposite to $\frac{\pi}{2q(m)}$ is still $\frac{1}{2}$.

In the following, let $v$ be a vertex of $X'$. The \emph{link} of $v$ in $X'$ is the metric graph whose vertices correspond to the edges of $X'$ incident to $v$. Vertices of the link corresponding to edges $e_1,e_2$ of $X'$ are connected by an edge of length $\theta$, if $e_1$ and~$e_2$ lie in a common triangle $\sigma$ of $X'$ and form angle $\theta$ in $\sigma$. A \emph{loop} in the link is a locally embedded closed edge-path.

\begin{lemma}
\label{lem:truncated}
The truncated piecewise Euclidean structure satisfies the \emph{link condition}, i.e.\ each loop in the link of a vertex $v$ of $X'$ has length $\geq 2\pi$.
\end{lemma}

\begin{proof}[Proof of Lemma~\ref{lem:truncated}] If $v$ is the barycentre on an edge of $X$, then its link is a simple bipartite graph all of whose edges have length $\frac{\pi}{2}$. Hence its loops have length $\geq 4\frac{\pi}{2}=2\pi$.
If $v$ is the barycentre of a polygon of $X$, then its link is a circle that had length $2\pi$ in the usual piecewise Euclidean structure.
The angles at the barycentre of a polygon in the truncated Euclidean structure are at least as large as the angles in the usual piecewise Euclidean structure, and consequently in the truncated Euclidean structure the link has length $\geq 2\pi$. 

It remains to consider a vertex $v\in X^0$, and its link $L'$ in $X'$. For each triangle~$\sigma$ of $X'$ incident to $v$ there is exactly one other triangle $\tau$ of $X'$ incident to $v$ with common hypothenuse, and they lie in the same polygon of $X$. Let $L$ be the graph obtained from $L'$ by merging into one edge each pair of edges corresponding to such $\sigma$ and~$\tau$.
Note that $L$ is isometric to $L'$.
The graph $L$ has a vertex corresponding to each $s\in S$ and an edge of length $\big(1-\frac{1}{q(m_{st})}\big)\pi\geq \frac{\pi}{2}$ joining the vertices corresponding to $s,t$, for each $s,t\in S$ with $m_{st}<\infty$. In particular, all the loops in $L$ of combinatorial length $\geq 4$ have metric length $\geq 2\pi$. To obtain the same for loops in $L$ of combinatorial length $3$, we need to verify that for each triple of distinct $s,t,r\in S$, we have
\begin{equation*}
    \frac{1}{q(m_{st})}+\frac{1}{q(m_{tr})}+\frac{1}{q(m_{sr})}\leq 1.\tag{$*$}
\end{equation*}

If $q(m_{st}),q(m_{tr}), q(m_{sr})\neq 2$, then ($*$) holds. If $q(m_{st}),q(m_{tr})\neq 2$ and $q(m_{sr})=2$, then $m_{st},m_{tr}\neq 2$ and $m_{sr}=2$. Since $W$ is $2$-dimensional, we have $m_{st},m_{tr}\geq 4$ or $m_{st}\geq 6$ or $m_{tr}\geq 6$. We then have, respectively, $q(m_{st}),q(m_{tr})\geq 4$ or $q(m_{st})\geq 6$ or $q(m_{tr})\geq 6$, and so ($*$) holds in this case as well. Finally,
if $q(m_{st})=q(m_{sr})=2$, then $m_{st}=m_{sr}=2$, contradicting the $2$-dimensionality of $W$.
\end{proof}

Below, for two edges $e_1,e_2$ incident to a vertex $v$ of $X'$, by their \emph{angle} at $v$ we mean the distance in the link of $v$ between the vertices that $e_1,e_2$ correspond to. Since $X'$ satisfies the link condition, this is the same as the Alexandrov angle if the latter is $<\pi$.

\begin{lemma}
\label{lem:angles}
Let $W$ be a $2$-dimensional Coxeter group.
Let $\gamma, \gamma'$ be geodesic edge-paths in $X^1$
with common endpoints. Suppose that there are walls $\mathcal W_i$ in $X$ with $i=1,2,3,$ such that
$\gamma$ intersects them in the opposite order to $\gamma'$, and that $\mathcal W_2$ is the middle one in
both of these orders. For $i=1,3,$ let $\theta_i$ be the angle in the truncated structure at $x_i=\mathcal W_2\cap\mathcal W_i$
formed by the segments in $\mathcal W_2, \mathcal W_i$ from $x_i$ to $\gamma\cap \mathcal W_2$ and $\gamma\cap\mathcal W_i$.
Then $\theta_1+\theta_3<\pi$.
\end{lemma}

\begin{figure}[h!]
	\begin{center}
		\includegraphics[scale=0.63]{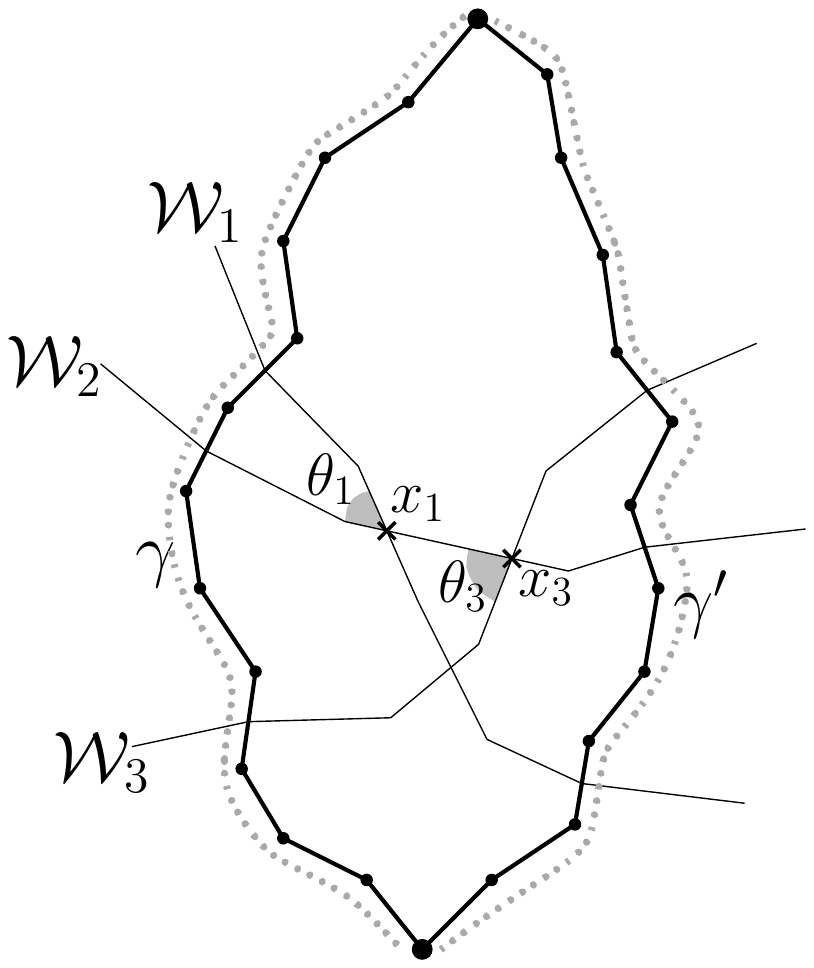}
	\end{center}
	\caption{Lemma~\ref{lem:angles}}\label{f:f1}
\end{figure}

See Figure~\ref{f:f1} for an illustration. Note that in the definition of either $\theta_i$ we could replace $\gamma$ by $\gamma'$.

In the proof we will need the following terminology. A \emph{combinatorial $2$-complex} is a $2$-dimensional CW complex in which the attaching maps of $2$-cells are closed edge-paths. For example, the Cayley complex $X$ of a Coxeter group is a combinatorial 2-complex. A \emph{disc diagram} $D$ is a compact contractible combinatorial 2-complex with a fixed embedding in $\R^2$. Its \emph{boundary path} is the attaching map of
the cell at $\infty$. If $X$ is a combinatorial 2-complex, \emph{a disc diagram in $X$} is a cellular map $\varphi \colon D\to X$ that is \emph{combinatorial}, i.e.\ its restriction to each cell of $D$ is a homeomorphism onto a cell of $X$.
The \emph{boundary path} of a disc diagram $\varphi \colon D\to X$ is the composition of the boundary path of $D$ and $\varphi$. We say that $\varphi$ is \emph{reduced}, if it is locally injective on $D-D^0$.

\begin{proof}[Proof of Lemma~\ref{lem:angles}]
Let $\varphi \colon D\to X$ be a reduced disc diagram in $X$ with boundary
$\gamma^{-1}\gamma'$ (for the existence of $\varphi$, see for example \cite[\S V.1--2]{LS}).
Consider the piecewise Euclidean structure on the barycentric subdivision $D'$ of $D$ that is the pullback under $\varphi$ of the truncated Euclidean structure on $X'$. By Lemma~\ref{lem:truncated} and \cite[II.5.4]{BH}, the induced path-metric on $D$ is $\mathrm{CAT}(0)$. Furthermore, for each wall $\mathcal W$ in $X$, the preimage $\varphi^{-1}(\mathcal W)$ is a geodesic in $D$.
Thus $D$ contains a geodesic triangle formed by the segments of  $\varphi^{-1}(\mathcal W_i)$ joining their three intersection points.
Its angles indicated in Figure~\ref{f:f1} equal $\theta_1,\theta_3$. Since the Alexandrov angles of that triangle do not exceed the angles of the comparison triangle in the Euclidean plane \cite[II.1.7(4)]{BHa}, we have $\theta_1+\theta_3<\pi$.
\end{proof}

\begin{cor}
\label{cor:step1}
Let $W$ be a $2$-dimensional Coxeter group.
Let $f\in W$ with $T(f)=\{s,t\}$, with $s\neq t$.
Let $h=\Pi(f)$ and let $R$ be the $\{s,t\}$-residue containing $f$ and $h$.
Let $g\in R$ and let $m$ be the distance in $X^1$ between $g$ and $h$. Suppose $T(g)=\{s,r\}$ with $r\neq s,t$.
Then:
\begin{enumerate}[(i)]
\item $m\leq 3$.
\item If $m=3$, then $m_{sr}=2$.
\item If $m_{st}=3$ and $m=2$, then $m_{sr}=2$.
\item If $m_{st}=4$, then $m\leq 2$.
\item If $m=m_{st}-1$, then $m_{st}\leq 3$, and for $m_{st}=3$ we have $m_{sr}=2$.
\end{enumerate}
\end{cor}

\begin{proof} Note that $T(g)=\{s,r\}$ implies in particular
$g\neq f,h$. Let $\gamma_0$ be the geodesic edge-path in $X^1$ from $f$ to $h$ not containing $g$. Let $\gamma_1$ be
the geodesic edge-path of length $m_{sr}$ with vertices in the $\{s,r\}$-residue containing $g$, starting at $g$ with the $r$-edge. Let $\gamma$ be any
geodesic edge-path from $f$ to $\id$ containing $\gamma_0$. Let $\gamma'$ be any geodesic edge-path from $f$ to $\id$ containing $\gamma_1$.
Let $\mathcal W_1$ be the first wall intersecting~$\gamma$. Let $\mathcal W_2$ be the wall dual to the $s$-edge incident to $g$.
Let $\mathcal W_3$ be the wall dual to the $r$-edge incident to $g$. See Figure~\ref{fig:corollary}. Note that $\mathcal W_3$ does not
intersect $R$ (since then $\mathcal W_2$ and $\mathcal W_3$ would intersect twice in $X$) and, analogously, $\mathcal W_1$ does
not intersect the $\{s,r\}$-residue of $g$. Consequently, we are in the setup of Lemma~\ref{lem:angles} and we let $\theta_1,\theta_3$ be as in that lemma, so that $\theta_1+\theta_3<\pi$.

\begin{figure}[h!]
	\begin{center}
		\includegraphics[scale=0.6]{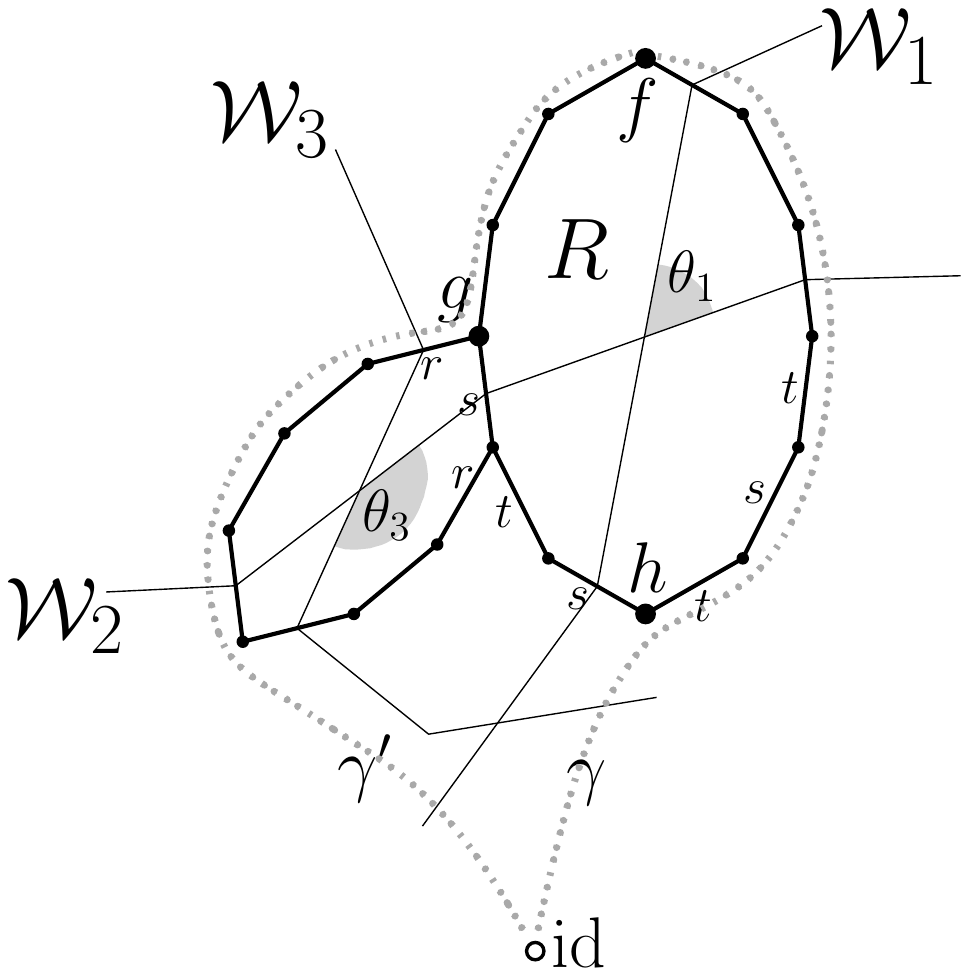}
	\end{center}
	\caption{Corollary~\ref{cor:step1}}\label{fig:corollary}
\end{figure}

Observe that we have $\theta_1=(m-1)\frac{\pi}{q(m_{st})}$ and $\theta_3=(m_{sr}-1)\frac{\pi}{q(m_{sr})}$.

To prove part (i), assume $m\geq 4$. We then have $\theta_1\geq \frac{\pi}{2}$. However, $\theta_3\geq \frac{\pi}{2}$, which
contradicts Lemma~\ref{lem:angles}.

For part (ii), if $m=3$ then we only have $\theta_1\geq \frac{\pi}{3}$. However, assuming $m_{sr}\geq 3$, we would have $\theta_3\geq \frac{2\pi}{3}$,
which also contradicts Lemma~\ref{lem:angles}.

For part (iii), if $m=2$ and $m_{st}=3$, then we have $\theta_1=
\frac{\pi}{3}$. Assuming $m_{sr}\geq 3$, we would have $\theta_3\geq \frac{2\pi}{3}$ as before,
which contradicts Lemma~\ref{lem:angles}.

To prove part (iv), if we had $m_{st}=4$ and $m\geq 3$, then $\theta_1\geq \frac{\pi}{2}$ and $\theta_3\geq
\frac{\pi}{2}$ would
also contradict Lemma~\ref{lem:angles}.

For part (v), assume $m=m_{st}-1$. The case $m_{st}\geq 5$ is excluded by part (i), and the case $m_{st}=4$ is excluded by part (iv). For $m_{st}=3$ we
have $m_{sr}=2$ by part (iii).
\end{proof}

\begin{proof}[Proof of Proposition~\ref{prop:main}]
If $s=t$, then without loss of generality $s\in T(g)$, and we
can take $k=1, k'=0$.

Assume now $s\neq t$. Let $R=g \langle s,t \rangle$, and let $f,h\in R$ be the elements
with maximal and minimal word length, respectively.
Let $m,m'$ be the distances in~$X^1$ between $h$ and
$g,g'$, respectively. We can assume $\Pi(g),\Pi(g')\notin R$. Then in
particular $m,m'\neq m_{st}$ and if $m\neq 0$ we have
$|T(g)|=2$ and $T(g)$ contains exactly one of $s,t$. Without loss of generality we suppose then
$T(g)=\{s,r\}$ for some $r\neq s,t$.

Note that from Corollary~\ref{cor:step1}(i) it follows that $m\leq 3$. Furthermore, by Corollary~\ref{cor:step1}(ii) if
$m=3$, then $m_{sr}=2$. An analogous statement holds for $m'$.

\smallskip

\noindent \textbf {Case 1: $m=3$, or $m=2$ and $m_{sr}\geq 3$.}

If $m=3$, then denoting by $\hat g$ the $s$-neighbour of $g$, we have $T(\hat g)=\{t,r\}$. Since $m_{sr}=2$, we have $m_{tr}\geq 3$.

\begin{figure}[h!]
	\begin{center}
		\includegraphics[scale=0.57]{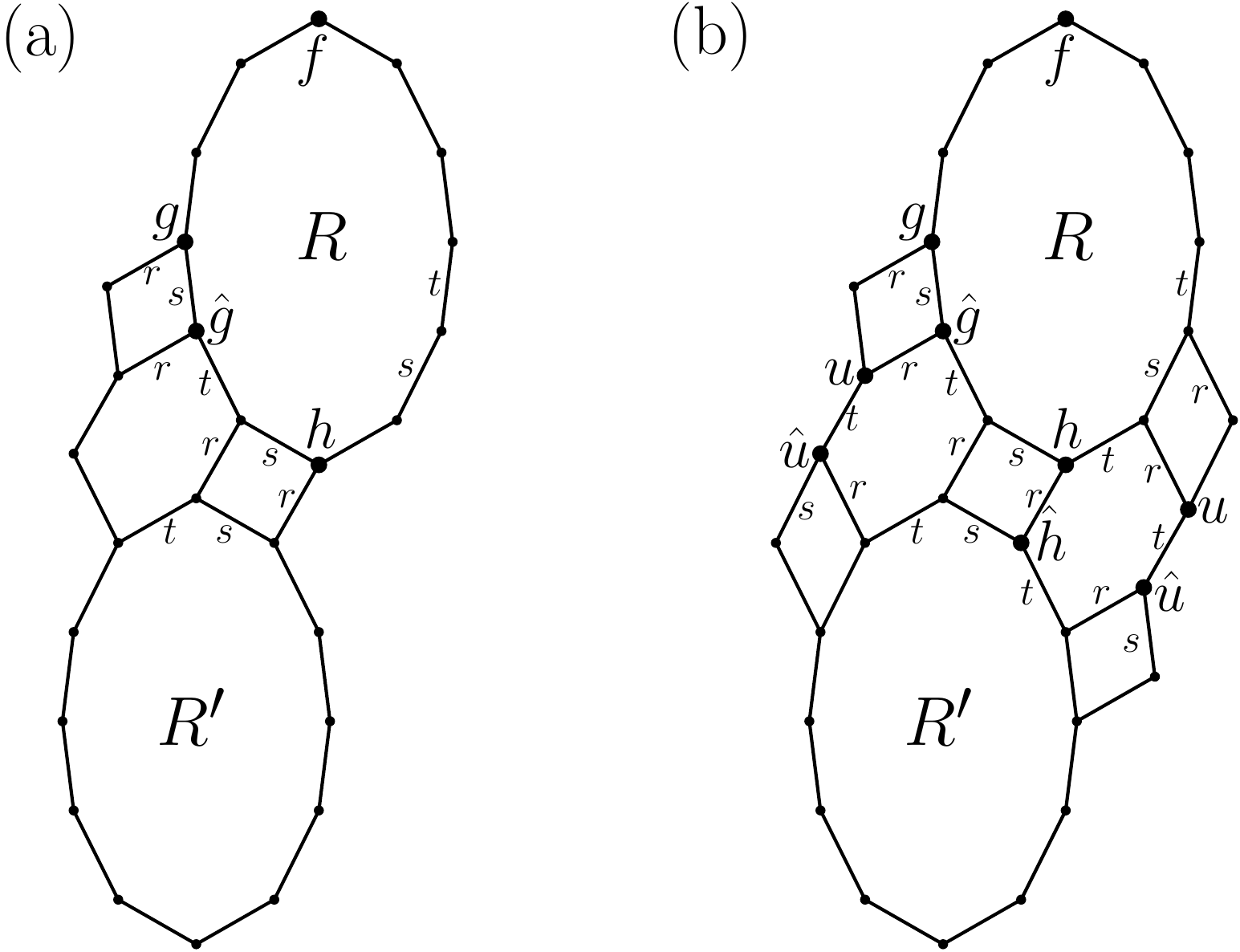}
	\end{center}
	\caption{Proof of Proposition~\ref{prop:main}, Case 1.}\label{fig:step2}
\end{figure}

Applying Corollary~\ref{cor:step1}(v), with $f$ replaced by $\hat g$ and $g$ replaced by the $t$-neighbour of $\hat g$, gives $m_{tr}=3$, and so $m_{st}\geq 6$.
Consequently, in $X^1$ we have the configuration described in Figure~\ref{fig:step2}(a),
where for each edge $(q,\hat q)$ of $X^1$ the vertex $q$ is drawn higher than $\hat q$ if $\ell(q)>\ell(\hat q)$.
For each edge-path in $X^1$ labelled $sr,rs,$ or $trt$, with endpoints $q,\hat q$ satisfying $\ell(q)-\ell(\hat q)=2$ or $3$, respectively, there is another edge-path from $q$ to $\hat q$ labelled $rs,sr,$ or $rtr$, respectively. The word lengths of the consecutive vertices of such a path
are $\ell(q),\ell(q)-1,\ell(q)-2=\ell(\hat q),$ or $\ell(q),\ell(q)-1,\ell(q)-2,\ell(q)-3=\ell(\hat q),$ respectively.
Thus the configuration described in Figure~\ref{fig:step2}(a) extends to the configuration in Figure~\ref{fig:step2}(b). In particular, we have $m'\neq 3$,
since otherwise, for $r'\in S$ satisfying $T(g')=\{t,r'\}$, denoting by $\hat{g}'$ the $t$-neighbour of~$g'$, we have $T(\hat{g}')=\{s,r'\}$ and so $r'=r$. Since $m_{tr'}=2$, we have $3\leq m_{sr'}=m_{sr}$, which is a contradiction. Consequently, $m'\leq 2$.

Consider any of the two vertices labelled by $u$ in Figure~\ref{fig:step2}(b). Note that $T(u)=\{t\}$,
since having $|T(u)|=2$ would force the $t$-neighbour $\hat u$ of $u$ to have $|T(\hat u)|\geq 3$.
This implies that $\Pi^3(g)$ lies on the lower $\{s,t\}$-residue $R'$ in Figure~\ref{fig:step2}(b).
Furthermore, note that $T(h)=\{r\}$, since having $T(h)=\{r,p\}$ for some $p\in S$
would force the $r$-neighbour $\hat h$ of $h$ to have $T(\hat
h)=\{t,p\}$, contradicting Corollary~\ref{cor:step1}(v) with $g$ replaced by $\hat
h$, and $f$ replaced by the $s$-neighbour of $\hat h$, since it would imply $m_{st}\leq 3$.
Consequently, in any of the cases $m'=0,1,2$, there is
$k'\leq 3$ with $\Pi^{k'}(g')\in R'$, as desired.

If $m=2$ and $m_{sr}\geq 3$, then the same proof goes through with
the following minor changes. Namely, $m_{sr}=3$ and $m_{tr}=2$ follow from Corollary~\ref{cor:step1}(v)
applied with $f$ replaced by $g$ and $g$ replaced by the $s$-neighbour of $g$. The remaining part of the
proof is the same, with $s$ and $t$ interchanged, except that it is $\Pi(g)$ instead of $\Pi^3(g)$ that lies in $R'$.
Namely, in $X^1$ we have the configuration described in Figure~\ref{fig:step2}(a), with the top square removed, $s$ and $t$ interchanged, and $\hat g$ replaced with $g$. Consequently, we have the configuration described in Figure~\ref{fig:step2}(b), with the same modifications. We obtain $k'\leq 3$ with $\Pi^{k'}(g')\in R'$ as before.

\smallskip

\noindent \textbf {Case 2: $m=2$ and $m_{sr}=2$.}

We have $m_{tr}\geq 3$ and the configuration from Figure~\ref{fig:step3} inside $X^1$. Note that if $m'=2$, then we can assume $T(g')=\{t\}$. Indeed, if $T(g')=\{t,p\}$, then we can assume $m_{tp}=2$ since otherwise interchanging $g,g'$ we can appeal to Case~1. Thus $p\neq r$, and so the $t$-neighbor $\hat {g}'$ of $g'$ has $|T(\hat {g}')|\geq 3$, which is a contradiction.

\begin{figure}[h!]
	\begin{center}
		\includegraphics[scale=0.57]{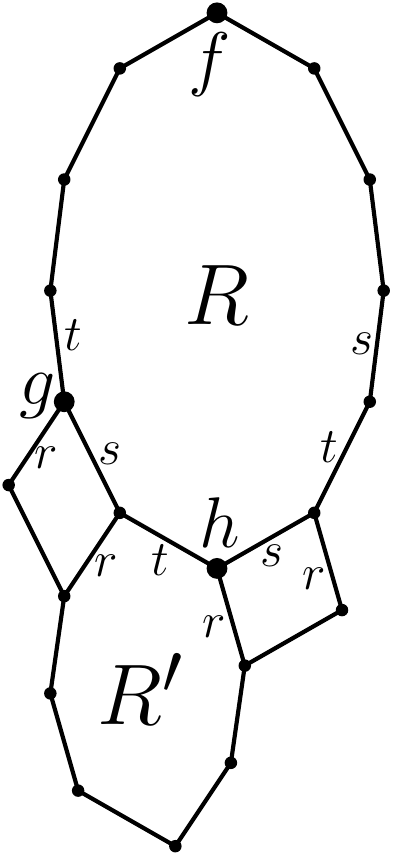}
	\end{center}
	\caption{Proof of Proposition~\ref{prop:main}, Case 2.}\label{fig:step3}
\end{figure}

Consequently both $\Pi(g)$ and $\Pi^{k'}(g')$ for some $k'\leq 2$ lie in the $\{t,r\}$-residue $R'$ from Figure~\ref{fig:step3}. This completes Case~2.

\smallskip

Note that if, say, $m=1,m'=0$, then we can take $k=1,k'=0.$
Thus it remains to consider the case where $m'=m=1$.

\smallskip

\noindent \textbf {Case 3: $m'=m=1$, and $T(g')=\{t,r\}.$} In other words, the second element of $T(g')$ coincides with that of $T(g)$.

If one of $m_{sr},m_{tr}$, say $m_{sr}$, equals $2$, then we can take $k=1,k'=0$, and we are done. If $m_{sr}=m_{tr}=3$, then we can take $k=k'=1$. It remains to consider the case, where, say, $m_{sr}\geq 4, m_{tr}\geq 3$.
Let $\gamma_0,\gamma_0'$ be the geodesic edge-paths from $f$ to $g,g'$, respectively.
If $m_{tr}\geq 4$, then we apply Lemma~\ref{lem:angles} with any $\gamma$ starting with $\gamma_0\overbrace
{rsr\cdots}^{m_{sr}}$, and any $\gamma'$ starting with $\gamma'_0\overbrace
{rtr\cdots}^{m_{tr}}$. We take $\mathcal W_1,\mathcal W_2,\mathcal W_3$ to be the walls dual to $r$-edges incident to $g,h,g',$ respectively. Then $\theta_1,\theta_3\geq \frac{\pi}{2}$, which is a contradiction. Analogously, if $m_{sr}\geq 6$, then $\theta_1\geq \frac{2\pi}{3},\theta_3\geq \frac{\pi}{3}$, contradiction.

We can thus assume $m_{tr}= 3,$ and $m_{sr}=4$ or $5$. In particular, $m_{st}\geq 3$. We now apply Corollary~\ref{cor:step1}, with $f$ replaced by
the $r$-neighbour $u$ of $h$ and $g$ replaced by the $s$-neighbour $\hat u$ of $u$, see Figure~\ref{fig:case3}. Since $T(u)=\{s,t\}$ with $m_{st}\geq 3$ and $T(\hat u)=\{t,r\}$ with $m_{tr}= 3$, Corollary~\ref{cor:step1}(v) yields a contradiction.

\begin{figure}[h!]
	\begin{center}
		\includegraphics[scale=0.57]{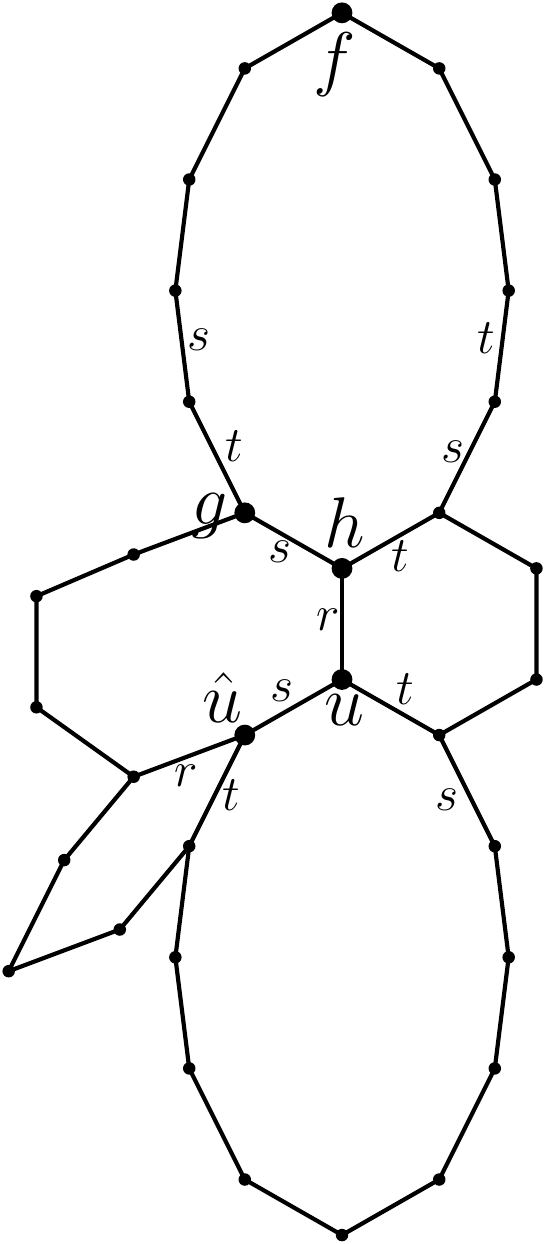}
	\end{center}
	\caption{Proof of Proposition~\ref{prop:main}, Case 3.}\label{fig:case3}
\end{figure}

\smallskip

\noindent \textbf {Case 4: $m'=m=1$, and $T(g')=\{t,p\}$ for some $p\neq r$.}

First note that $T(h)=\{r,p\}$ and so $m_{pr}<\infty$. If $m_{sr}=m_{tp}=2$, then we can take $k'=k=1$ and we are done. We now focus on the case $m_{sr}\geq 3$ and $m_{tp}\geq 3$.
By Corollary~\ref{cor:step1}(v), applied with $f$ replaced by~$g$ and $g$ replaced by $h$, we obtain $m_{sr}=3$ and $m_{rp}=2$. In particular, since $m_{tp}<\infty$,  we have $m_{tr}\geq 3$.
Let $\hat h$ be the $p$-neighbour of~$h$. We then apply Lemma~\ref{lem:angles} to geodesic edge-paths $\gamma,\gamma'$ from $g$ to $\id$, where $\gamma$ starts with the edge-path of length $m_{sr}$ in the $\{s,r\}$-residue of $g$ starting with the $r$-edge, and $\gamma'$ starts with the $s$-edge, the $p$-edge, followed by the edge-path of length $m_{tr}$ in the $\{t,r\}$ residue of $\hat h$ starting with the $t$-edge. See Figure~\ref{fig:case4}. We consider the walls $\mathcal W_1, \mathcal W_2$ dual to the $r$-edges incident to $g,h$, respectively, and $\mathcal W_3$ dual to the $t$-edge incident to $\hat h$. We have $\theta_1= \frac{\pi}{3}, \theta_3\geq \frac{2\pi}{3}$, which is a contradiction.

\begin{figure}[h!]
	\begin{center}
		\includegraphics[scale=0.6]{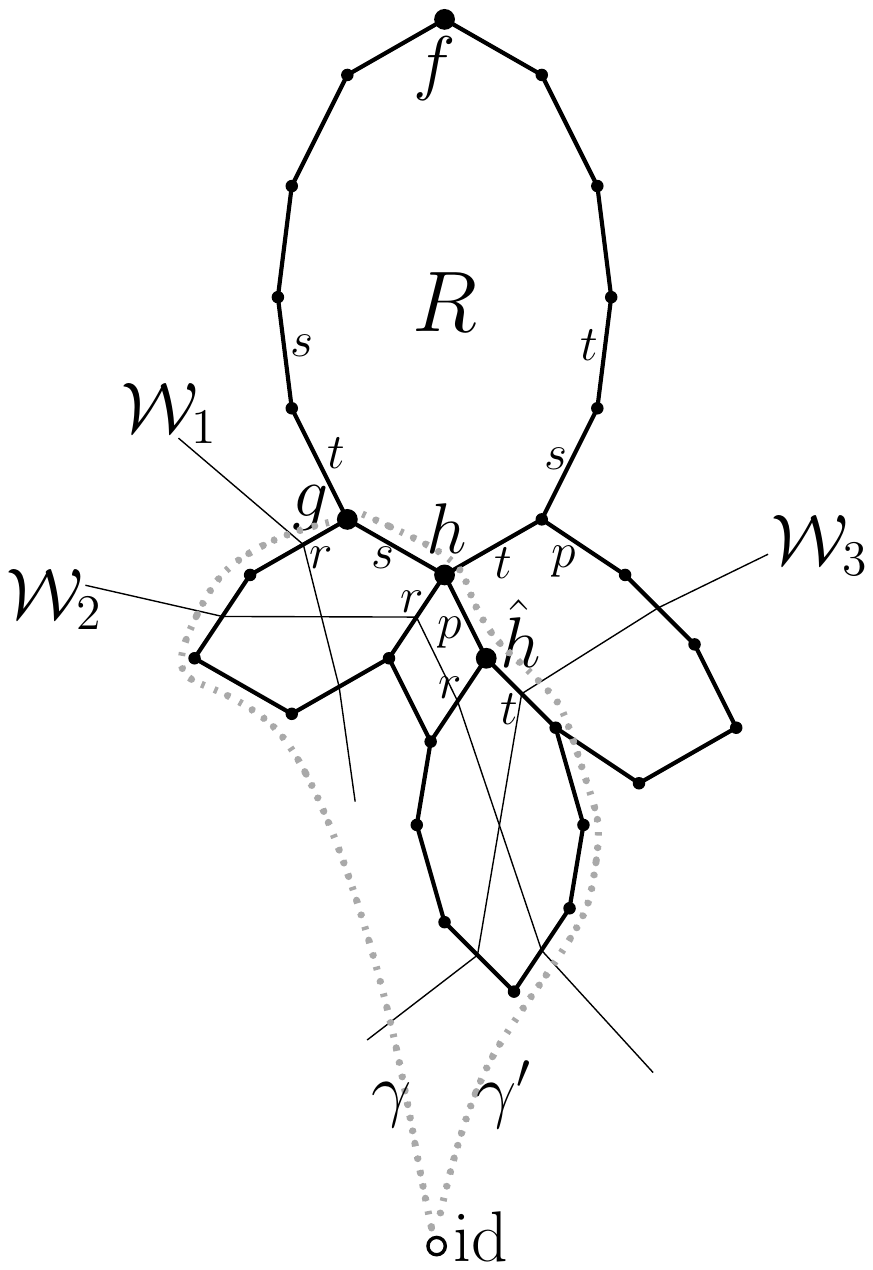}
	\end{center}
	\caption{Proof of Proposition~\ref{prop:main}, Case 4, $m_{sr}\geq 3$ and $m_{tp}\geq 3$.}\label{fig:case4}
\end{figure}

It remains to consider the case where, say, $m_{sr}\geq 3$ and $m_{tp}=2$. Then again by Corollary~\ref{cor:step1}(v), applied with $f$ replaced by $g$ and $g$ replaced by $h$, we obtain $m_{sr}=3$ and $m_{rp}=2$. Let $u$ be the $r$-neighbour of~$h$. Then $\Pi(g)$ lies in the $\{p,s\}$-residue $R'$ of $u$. Let $\hat h=\Pi(g')$, and let $\hat u$ be the $p$-neighbour of $u$, see Figure~\ref{fig:lastcase}. We have $m_{sp}\geq 6$ and so by Corollary~\ref{cor:step1}(v), applied with $f$ replaced by $u$ and $g$ replaced by $\hat u$, we obtain $T(\hat u)=\{s\}$. We claim that $T(\hat h)=\{r\}$ and so $\Pi(\hat h)$ also lies in $R'$, finishing the proof. To justify the claim, suppose $T(\hat h)=\{r,q\}$ with $q\neq r$. If $m_{rq}\geq 3$, then
we consider the walls $\mathcal W_1, \mathcal W_2$ dual to the $r$-edges incident to $g,h$, respectively, and $\mathcal W_3$ dual to the $q$-edge incident to $\hat h$, which leads to a contradiction as in the previous paragraph. If $m_{rq}=2$, then $q\neq s$ and so $T(\hat u)=\{s,q\}$, which is a contradiction.
This justifies the claim and completes Case~4.
\end{proof}

\begin{figure}[h!]
	\begin{center}
		\includegraphics[scale=0.63]{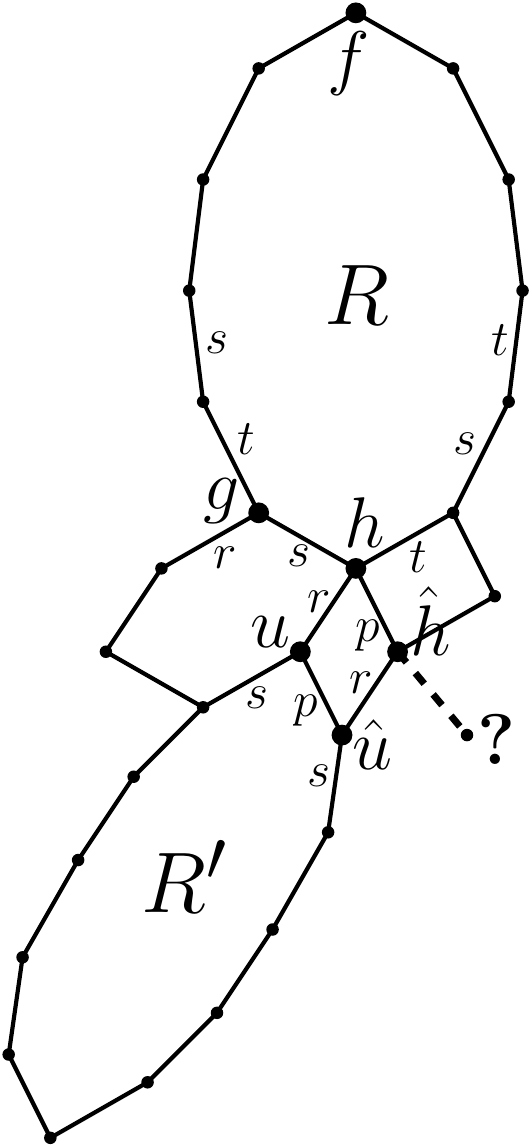}
	\end{center}
	\caption{Proof of Proposition~\ref{prop:main}, Case 4, $m_{sr}\geq 3$ and $m_{tp}=2$.}\label{fig:lastcase}
\end{figure}

\section{$\widetilde A_3$ Euclidean group}
\label{sec:dim3}

In this section it will be convenient to view the Cayley graph $X^1$ of the $\widetilde A_3$ Coxeter group $W$ as the dual graph to its Coxeter complex, which is the following subdivision of $\R^3$. (The reader might find it convenient to relate this subdivision intro tetrahedra with the standard subdivision of $\R^3$ into unit cubes.) Its vertices are triples of integers $(x,y,z)$ that are all odd or all even. Edges connect each vertex $(x,y,z)$ to vertices of the form $(x\pm 2,y,z),(x,y\pm 2,z),(x,y,z\pm 2),(x\pm 1,y\pm 1,z\pm 1)$, where the three signs can be chosen independently. See for example \cite[Thm~A]{Munro}, where this Coxeter complex is described as a subdivision of the hyperplane $\sigma$ in $\R^4$ defined by $x_1+x_2+x_3+x_4=0$, and the linear isomorphism with our subdivision of $\R^3$ is given by $(x,y,z)\mapsto (x+y+z,x-y-z, y-z-x, z-x-y)$. Furthermore, in Step~1 of the proof of \cite[Thm~A]{Munro}, we show that the tetrahedra of the Coxeter complex are obtained by subdividing $\sigma$ along a family of hyperplanes that, after identifying $\sigma$ with $\R^3,$ have equations $x\pm y=c, x\pm z=c,$ or $ y\pm z=c$, for $c$ even. 

In particular, in the second paragraph of Step~1 of the proof of \cite[Thm~A]{Munro}, we describe explicitly one of the tetrahedra as, after identifying $\sigma$ with $\R^3$, spanned on the clique with vertices 
$(-1,-1,-1),(-1,-1,1),(-2,0,0),(0,0,0)$. Using the action of $W$, this gives the following description of all the tetrahedra of our subdivision.
Namely, tetrahedra are spanned (up to permuting the coordinates) on cliques with vertices $(x,y,z-1),(x,y,z+1),(x+1,y-1,z),(x+1,y+1,z)$.
Each such tetrahedron has exactly two edges of length $2$, and the segment $e=((x,y,z),(x+1,y,z))$ joining their centres has length $1$. We can equivariantly embed $X^1$ into $\R^3$ by mapping each vertex into the centre of a tetrahedron, and mapping each edge affinely. Consequently, we can identify elements $g\in W$ with segments of the form $e_g=((x,y,z),(x+1,y,z))$, where $y+z$ is odd, up to permuting the coordinates. We identify $\id\in W$ with $e_{\id}=((0,0,1),(0,1,1))$. In particular, the point $O=(0,0,0)$ belongs to the identity tetrahedron. Note that for each $g\in W, s\in S$, the segments $e_g,e_{gs}$ are incident. Furthermore, walls in $X^1$ extend to subcomplexes of $\R^3$ isometric to Euclidean planes, and such a wall is adjacent to $g\in W$ if and only if it contains a face of the tetrahedron containing $e_g$.

\begin{lemma}
\label{lem:stepin A3}
Let $|x_0|+1<y_0<z_0$. Let $g\in W$ be such that
\begin{enumerate}[(i)]
\item $e_g=((x_0,y_0,z_0),(x_0+1,y_0,z_0))$, or
\item $e_g=((x_0,y_0,z_0),(x_0,y_0,z_0+1))$.
\end{enumerate}
Then $\ell(w(g))$ equals, respectively,
\begin{enumerate}[(i)]
\item $3$, or
\item $2$.
\end{enumerate}
Furthermore, $e_{\Pi(g)}$ is equal to the translate of $e_g$ by, respectively,
\begin{enumerate}[(i)]
\item $(0,-1,-1)$, or
\item $(0,0,-1)$.
\end{enumerate}
\end{lemma}

\begin{proof} In case (i), suppose first that $x_0+y_0$ is even. Then $e_g$ lies in the tetrahedron with vertices $(x_0,y_0,z_0-1),(x_0,y_0,z_0+1),(x_0+1,y_0-1,z_0),(x_0+1,y_0+1,z_0)$. The walls adjacent to~$g$ are the hyperplanes containing the faces of this tetrahedron, which are $x+y=x_0+y_0, x-y=x_0-y_0, x+z=x_0+1+z_0, x-z=x_0+1-z_0$. Projecting $e_g,O,$ and these walls onto the $xy$ plane (Figure~\ref{fig:proj1}(a)), or the $xz$ plane (Figure~\ref{fig:proj1}(b)), we obtain that $e_g$ is separated from $O$ exactly by the first and fourth among these walls.

\begin{figure}[h!]
	\begin{center}
		\includegraphics[scale=0.63]{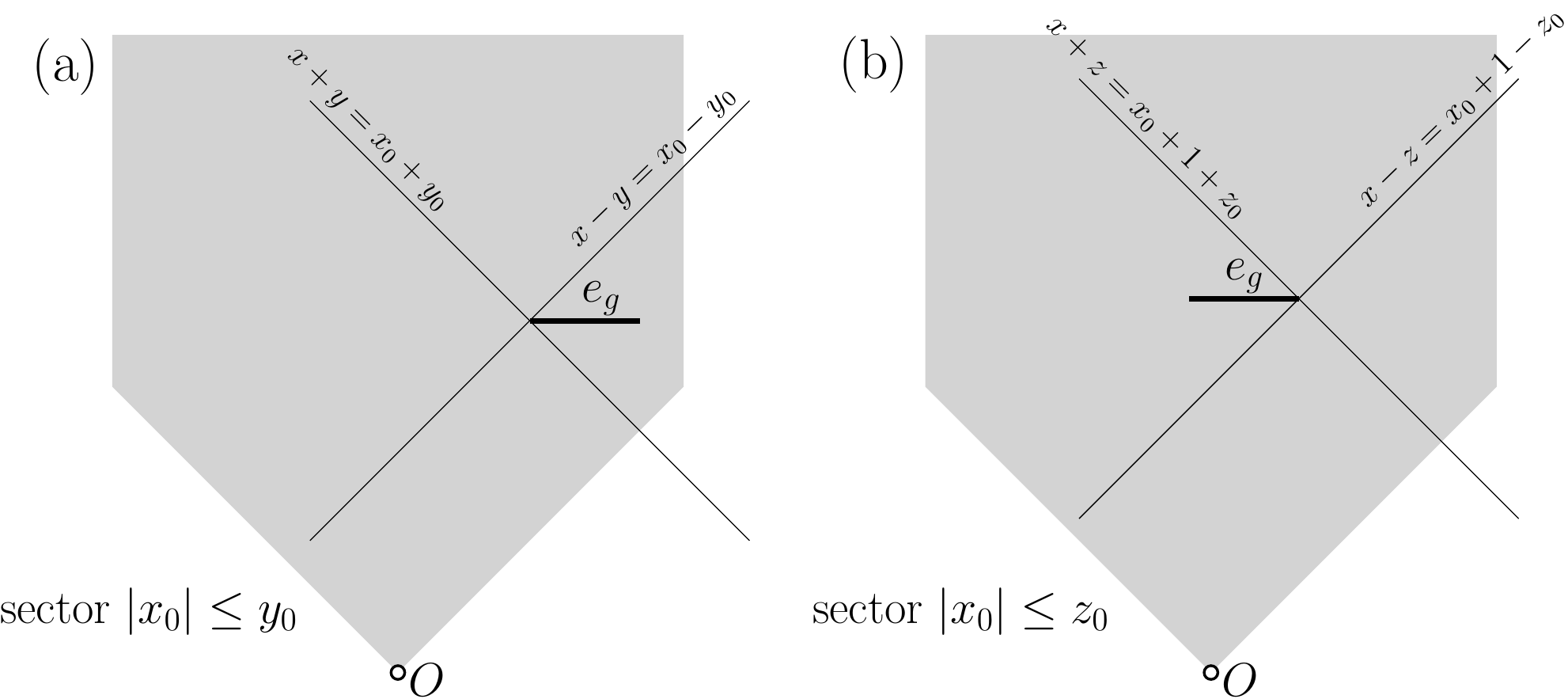}
	\end{center}
	\caption{Proof of Lemma~\ref{lem:stepin A3}, case (i), $x_0+y_0$ even.}\label{fig:proj1}
\end{figure}

\begin{figure}[h!]
	\begin{center}
		\includegraphics[scale=1]{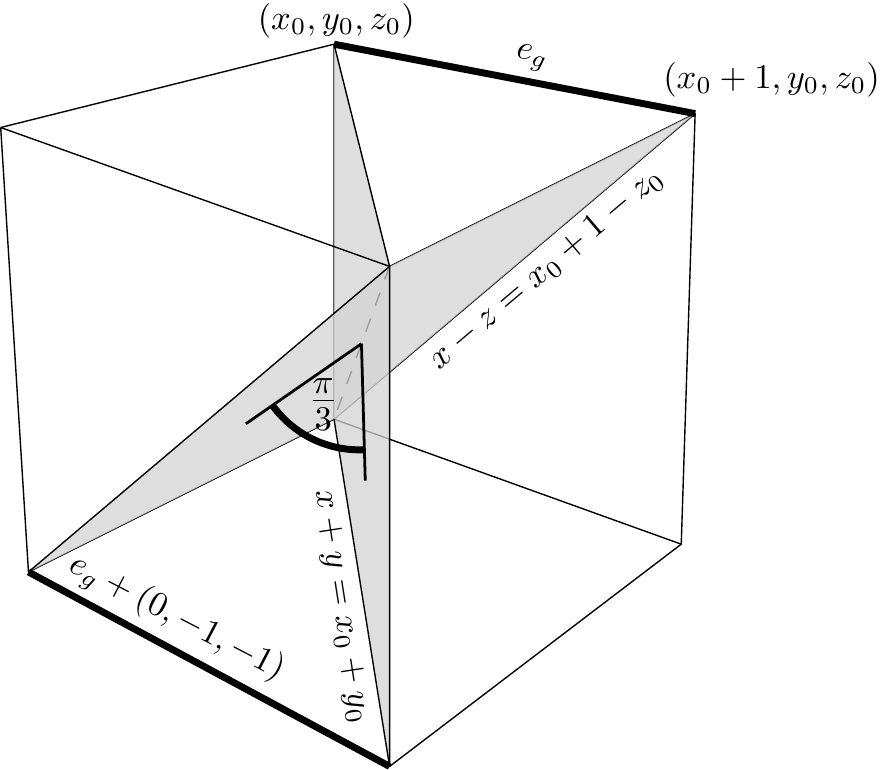}
	\end{center}
	\caption{Proof of Lemma~\ref{lem:stepin A3}, case (i), $x_0+y_0$ even: the two walls.}\label{fig:cube}
\end{figure}

Consequently, $gT(g)g^{-1}$ consists of the reflections in the first and fourth of these walls. These reflections preserve the cube spanned by $e_g$ and its translates by $(0,-1,0),(0,0,-1),$ and $(0,-1,-1)$, see Figure~\ref{fig:cube}. The longest element (of length~$3$) in the group that these reflections generate maps $e_g$ to its translate by $(0,-1,-1)$.

Secondly, suppose that $x_0+y_0$ is odd. Then $e_g$ lies in the tetrahedron with vertices $(x_0,y_0-1,z_0),(x_0,y_0+1,z_0),(x_0+1,y_0,z_0-1),(x_0+1,y_0,z_0+1)$. Thus the walls adjacent to $g$ are $x+y=x_0+1+y_0, x-y=x_0+1-y_0, x+z=x_0+z_0, x-z=x_0-z_0$. Hence, as illustrated in Figure~\ref{fig:proj2}(a,b), $e_g$ is separated from $O$ exactly by the second and third among these walls.
Consequently, $gT(g)g^{-1}$ consists of the reflections in the second and third of these walls. The longest element (of length $3$) in the group they generate maps $e_g$ to its translate by $(0,-1,-1)$ as before.

\begin{figure}[h!]
	\begin{center}
		\includegraphics[scale=0.63]{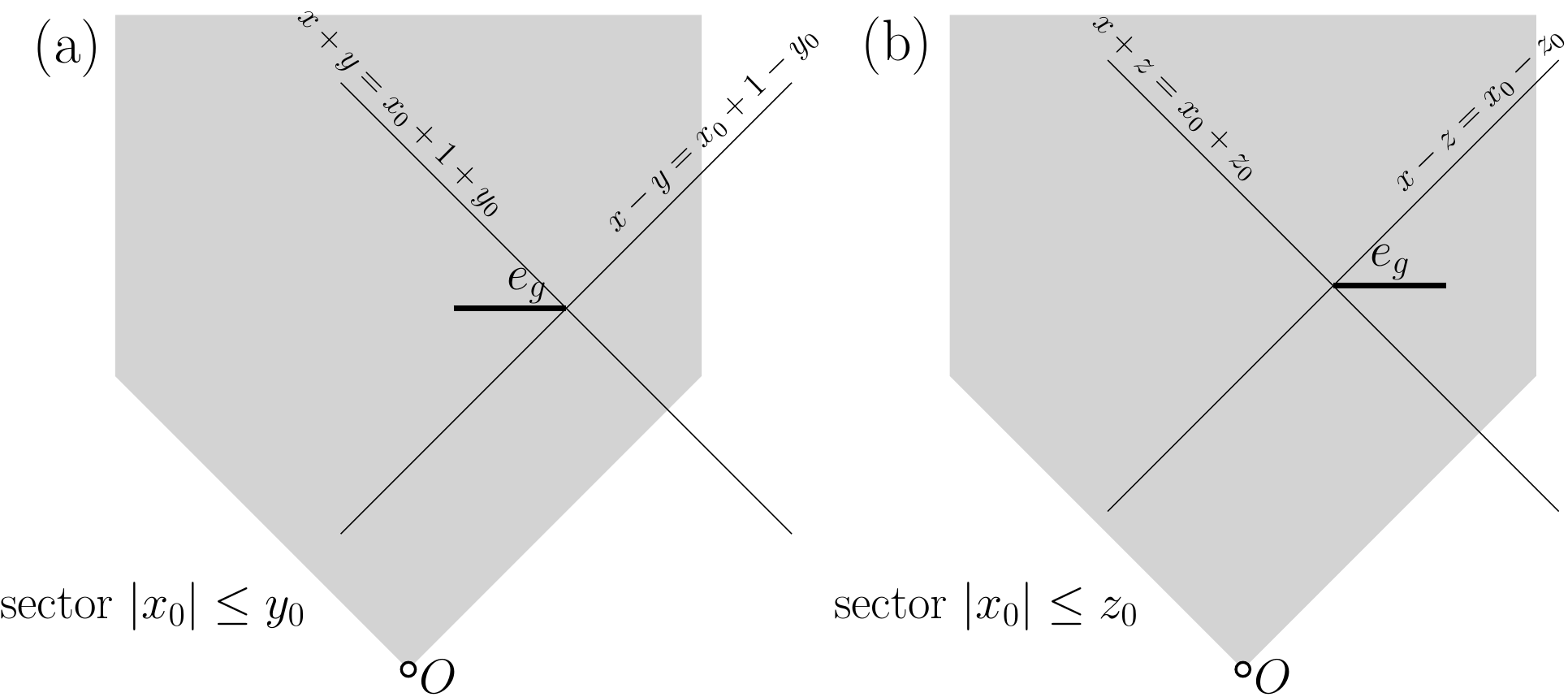}
	\end{center}
	\caption{Proof of Lemma~\ref{lem:stepin A3}, case (i), $x_0+y_0$ odd.}\label{fig:proj2}
\end{figure}

In case (ii), suppose first that $y_0+z_0$ is odd. Then $e_g$ lies in the tetrahedron with vertices $(x_0,y_0-1,z_0),(x_0,y_0+1,z_0),(x_0-1,y_0,z_0+1),(x_0+1,y_0,z_0+1)$.
Thus the walls adjacent to $g$ are $x+z=x_0+z_0, x-z=x_0-z_0, y+z=y_0+z_0+1, y-z=y_0-z_0-1$.
Hence, as illustrated in Figure~\ref{fig:proj3}(a,b), $e_g$ is separated from $O$ exactly by the first and second among these walls.
Consequently, $gT(g)g^{-1}$ consists of the reflections in the first and second of these walls. These reflections commute and preserve the square spanned by $e_g$ and its translate by $(0,0,-1)$. The longest element in the group these reflections generate (i.e.\ their composition) maps $e_g$ to its translate by $(0,0,-1)$.

\begin{figure}[h!]
	\begin{center}
		\includegraphics[scale=0.63]{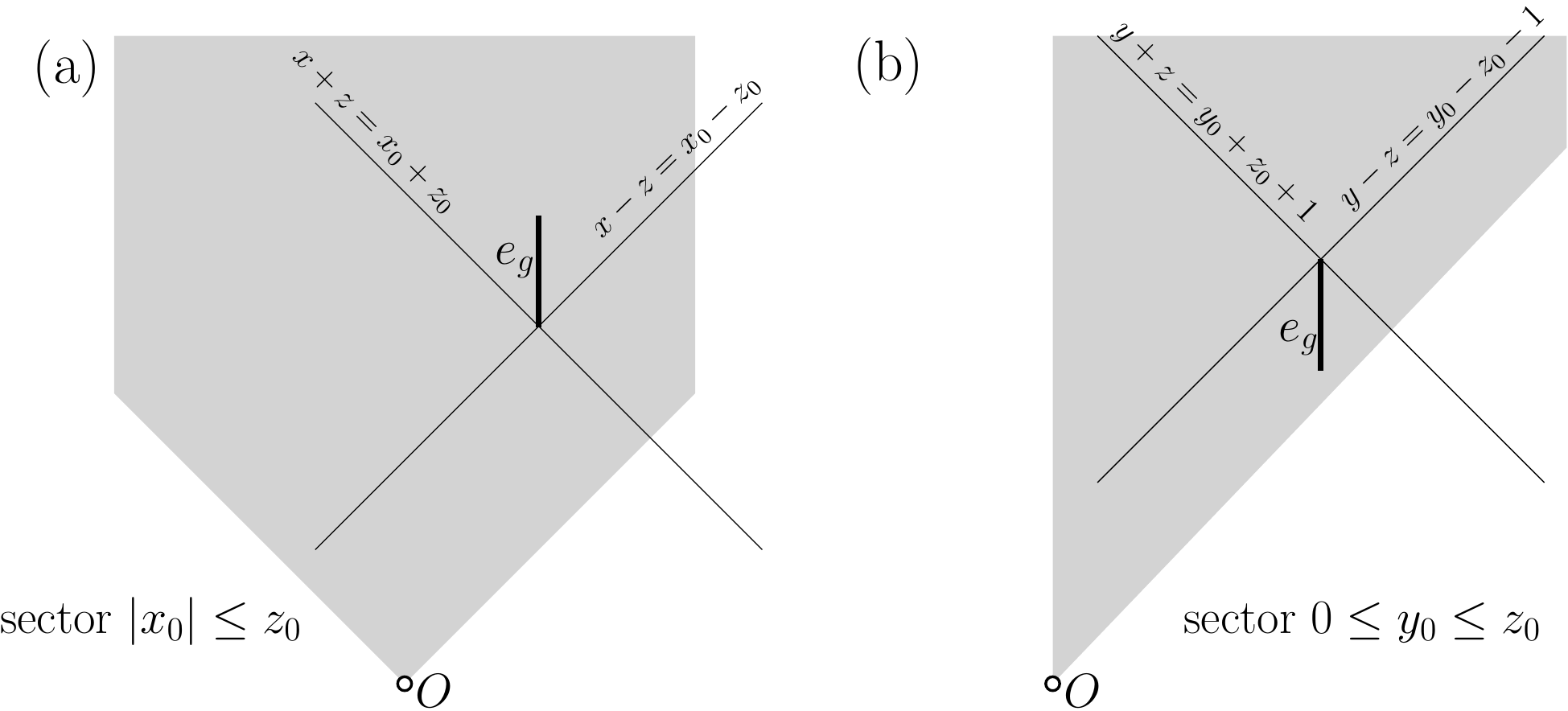}
	\end{center}
	\caption{Proof of Lemma~\ref{lem:stepin A3}, case (ii), $y_0+z_0$ odd.}\label{fig:proj3}
\end{figure}

\begin{figure}[h!]
	\begin{center}
		\includegraphics[scale=0.63]{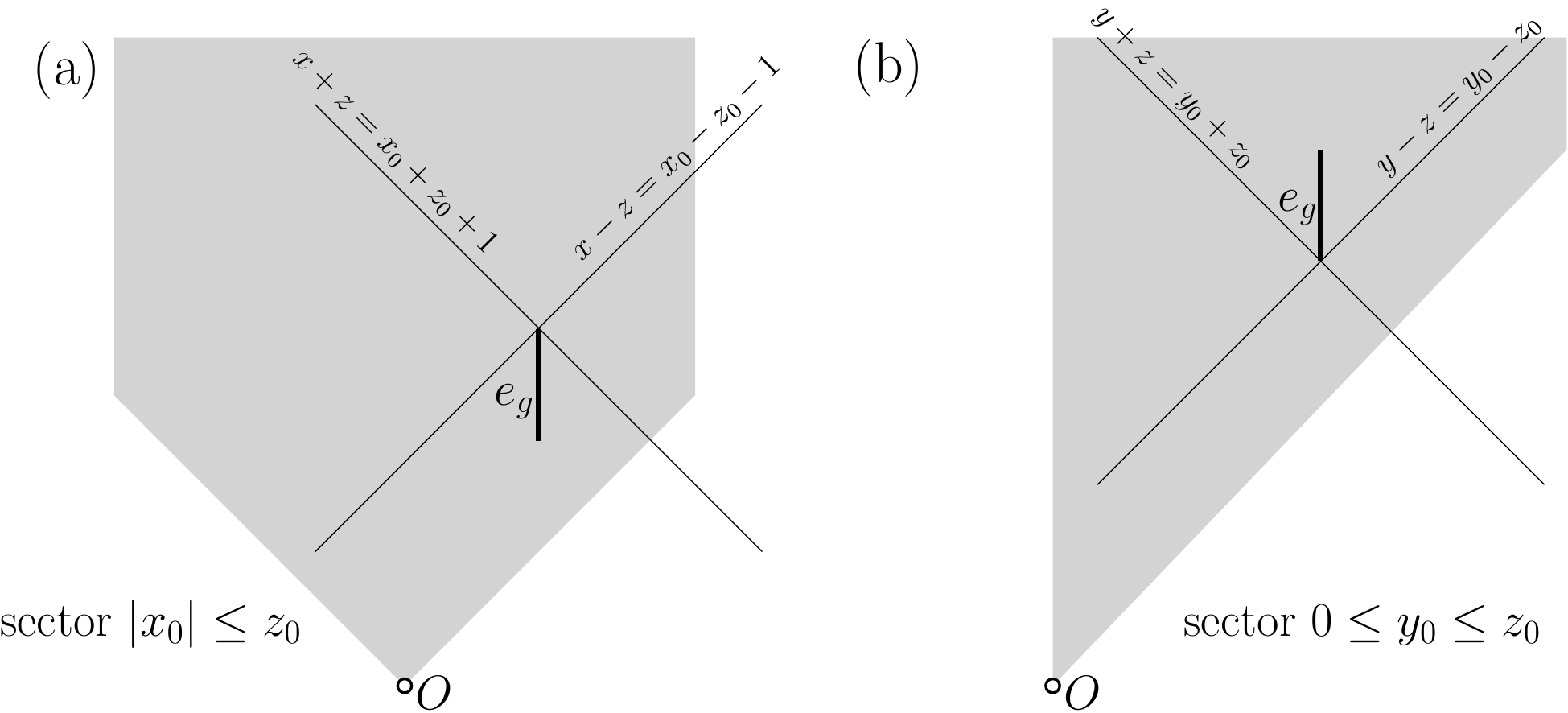}
	\end{center}
	\caption{Proof of Lemma~\ref{lem:stepin A3}, case (ii), $y_0+z_0$ even.}\label{fig:proj4}
\end{figure}

Secondly, suppose that $y_0+z_0$ is even. Then $e_g$ lies in the tetrahedron with vertices $(x_0-1,y_0,z_0),(x_0+1,y_0,z_0),(x_0,y_0-1,z_0+1),(x_0,y_0+1,z_0+1)$.
Thus the walls adjacent to $g$ are $x+z=x_0+z_0+1, x-z=x_0-z_0-1, y+z=y_0+z_0, y-z=y_0-z_0$.
Hence, as illustrated in Figure~\ref{fig:proj4}(a,b), $e_g$ is separated from $O$ exactly by the third and
fourth among these walls.
Consequently, $gT(g)g^{-1}$ consists of the (commuting)
reflections in the third and fourth of these walls.
The longest element in the group they generate
maps~$e_g$ to its translate by $(0,0,-1)$ as before.
\end{proof}

\begin{proof}[Proof of Theorem~\ref{thm:second}] Let $\mathcal L$ be the standard language.
For each $C>0$ consider the following $g,g'\in W$ with incident segments
$$e_g=((x_0,y_0,z_0),(x_0+1,y_0,z_0)),\ e_{g'}=((x_0,y_0,z_0),(x_0,y_0,z_0+1))$$
with $x_0,z_0$ even and $y_0$ odd, satisfying
$|x_0|+C<y_0\leq z_0-C$. Suppose that
$g,g'$ are represented by $v,v'\in\mathcal L$ of length $N,N'$ (which differ by $1$).
By Lemma~\ref{lem:stepin A3}, for $n,n'\leq C$ we have that $v(N-3n)$ represents the element of $W$ corresponding to the segment
$e_g-n(0,1,1)$ and $v'(N'-2n')$ represents the element of $W$ corresponding to the segment
$e_{g'}-n'(0,0,1)$. In particular, for $i=3n=2n',$ we see that the
segments corresponding to $v(N-i)$ and $v'(N'-i)$ are
$((x_0,y_0-n,z_0-n),(x_0+1,y_0-n,z_0-n))$ and $
((x_0,y_0,z_0-\frac{3}{2}n),
(x_0,y_0,z_0+1-\frac{3}{2}n))$.
Thus they are at Euclidean distance $\geq n$, so in particular
$\ell(v(N-i)^{-1}v'(N'-i))\geq n$. This shows that part
(ii) of the definition of biautomaticity does not hold for
$\mathcal L$.
\end{proof}

\end{document}